\documentclass{amsart}
\usepackage{amsthm}

\usepackage{amsmath}
\usepackage{amsfonts}

\usepackage{amssymb}
\newtheorem{theorem}{Theorem}
\newtheorem{proposition}[theorem]{Proposition}
\newtheorem{lemma}[theorem]{Lemma}

\begin{document}
\title{A family of cyclic quartic fields with explicit fundamental units}
\author{Steve Balady}
\address{Department of Mathematics,
University of Maryland,
College Park, MD 20742}
\email{sbalady@math.umd.edu} 
\author{Lawrence C. Washington}
\address{Department of Mathematics,
University of Maryland,
College Park, MD 20742}
\email{lcw@math.umd.edu}

\begin{abstract} We construct a family of quartic polynomials with cyclic Galois group and show that the roots of the 
polynomials are fundamental units or generate a subgroup of index 5.\end{abstract}
\maketitle
The goal of this paper is to prove the following.
\begin{theorem}\label{main} Let $s$ be an integer such that $3s^2-4s+4$ is a square.
Let $K_s$ be the splitting field of 
$$
F_s(t)= t^4 + (4s^3 - 4s^2 + 8s - 4)t^3 + (-6s^2 - 6)t^2 + 4t + 1,
$$
 Then $\text{Gal}(K_s/\mathbb Q)$ is cyclic of order 4.
If  $s^2+2$ is squarefree and $s\ne 0$, then  $\pm 1$ and the roots of $F_s(t)$  generate either the unit group
of the ring of algebraic integers in $K_s$ or a subgroup of index 5. 
\end{theorem}
Computational evidence (see the end of Section \ref{fundamental}) indicates that the index 5 case does not occur, but we have not yet proved this.

Families of cyclic quartic fields with explicit units have been studied in the past 
(for example, \cite{Gras}, \cite{Lecacheux1}, \cite{Lecacheux2}), but it does not
seem that this
family has been studied previously. In \cite{Balady}, families of cyclic cubic fields were constructed and the method led 
naturally to studying integral points on a model of the elliptic modular surface
$X(3)$. In the present case, we are led to study integral points on a degree 4 cover of the surface $X(2)$,
but the Diophantine properties are not as transparent. The family we study lives above a singular fiber of
$X(2)$. It would be interesting to know if there are other families living as curves on the surface.

\section{The polynomials}

As in \cite{Balady}, we start with the action of the Galois group given by a linear fractional transformation, 
but this time we take the matrix to have order 4 in 
$\text{PGL}_2(\mathbb Z)$. Let $f, g$ be integers and let
$$
M=\begin{pmatrix} f & -1\\ \frac{f^2+g^2}{2} & -g\end{pmatrix}.
$$
Then 
$$
M^2=\frac{f-g}{2}\begin{pmatrix} f+g & -2\\ f^2+g^2 & -g-f\end{pmatrix},\quad 
M^3=\frac{(f-g)^2}{2} \begin{pmatrix} g & -1\\ \frac{f^2+g^2}{2} & -f\end{pmatrix}.
$$
The action of the Galois group is to be $\theta \mapsto M\theta$. We want $\theta$ to be a unit,
so we take it to have norm 1:
\begin{equation}\label{norm1}
\theta\cdot \frac{f\theta -1}{\frac{f^2+g^2}{2}\theta - g}\cdot \frac{(f+g)\theta -2}{(f^2+g^2)\theta - g-f}\cdot
\frac{g\theta - 1}{\frac{f^2+g^2}{2}\theta - f}=1.
\end{equation}
If $fg(f+g)\ne 0$, this relation can be rearranged to say that $\theta$ is a root of
\begin{align}\label{fgpoly}
t^4 -& \frac{(f^2+g^2)^3+4(f^2+g^2)+16fg}{4fg(f+g)}t^3 \\
&+\frac{3((f^2+g^2)^2+4)}{4fg}t^2 - \frac{(f+g)^4-4f^2g^2+4}{2fg(f+g)}t + 1 = 0.
\end{align}
Since this is symmetric in $f$ and $g$, we make the substitutions $s=f+g$ and $p=fg$ to obtain
$$
t^4-\frac{(s^2-2p)^3+4s^2+8p}{4sp}t^3+\frac{3((s^2-2p)^2+4)}{4sp}t^2-\frac{s^4-4p^2+4}{2sp}+1.
$$
Let $L=-\frac{s^4-4p^2+4}{4sp}$. The polynomial becomes
\begin{equation}\label{polywithL}
t^4+(2s^3+Ls^2-4ps+2Lp)t^3+(-3s^2-3Ls+6p)t^2+2Lt+1.
\end{equation}
Therefore, if $L\in \mathbb Z$, or if $2L\in \mathbb Z$ and $s$ is even, the roots of the polynomial are units.
Of course, since $s=f+g$ and $p=fg$, there is the extra condition that $s^2-4p$ is a square.
\medskip

\noindent
{\bf Remark.} The divisibility condition $4sp\mid s^4-4p^2+4$ that makes $L$ integral is the same type of condition that appears
in the cubic case, where the paper \cite{Balady} has the condition $fg\mid f^3+g^3+1$. These both seem to be analogues
of the condition $r\mid 4n$ for the Richaud-Degert real quadratic fields $\mathbb Q(\sqrt{n^2+r})$ (see \cite{MollWill}).
\medskip

Note that changing $s$ to $-s$ while holding $p$ constant corresponds to changing the signs of $f$ and $g$. 
This has the effect of
$$
L\mapsto -L, \quad t\mapsto -t
$$
in (\ref{polywithL}). Therefore, we can for simplicity assume that $L>0$ (the case $L=0$ cannot occur).

The family considered in \cite{Gras} (the ``simplest quartic fields'') corresponds to the matrix $M$ with $f=1$ and $g=-1$.
However, the relation (\ref{norm1}) becomes trivial in this case and does not yield the polynomial defining the family.

\section{The surface}

When the analogous construction was done for the cubic case in \cite{Balady}, the expression
for one of the coefficients yielded an equation for $X(3)$. In the present case,
we are looking for integral points $(f, g, L)$ on the surface
$$
X:\quad (f+g)^4-4f^2g^2+4+2Lfg(f+g)=0.
$$
This is a double cover of the surface
$$
s^4-4p^2+4+2Lps=0,
$$
which can be transformed (see, for example, \cite{Connell}) to
$$
y^2=(x+4)(x-4)(x+L^2).
$$
Therefore, $X$ is a degree 4 cover of the elliptic surface $X(2)$, whose fibers are Legendre elliptic curves.
The bad fibers $L=\pm 2$ are the ones that play a role in what follows.

A computer search produced several pairs $(f, g)$ for which $L$ is integral, and almost all of them
had $L=\pm 2$. 

\begin{table}[h]
\begin{tabular}{c|c|c|c|c|c}
$f$ & $g$ & $s$ & $p$ & $L$ & polynomial\\
\hline
$1$ & $-5$ & $-4$ & $-5$ & $-2$ & $t^4 - 220t^3 - 102t^2 - 4t + 1$\\
5 & $-17$ & $-12$ & $-85$ & 2 & $t^4 - 7588t^3 - 870t^2 + 4t + 1$\\
5 & $-37$ & $-32$ & $-185$ & $-77/2$ & $t^4 - 114395t^3 - 7878t^2 - 77t + 1$\\
17 & $-65$ & $-48$ & $-1105$ & $-2$ & $t^4 - 433532t^3 - 13830t^2 - 4t + 1$\\
65 & $-241$ & $-176$ & $-15665$ & $2$ & $t^4 - 21932420t^3 - 185862t^2 + 4t + 1$
\end{tabular}
\end{table}
As mentioned above, a simple transformation changes the examples with $L=-2$ into examples with $L=2$.

\section{A family of fields}

From now on, we make the assumption:
$$
L= 2.
$$
Since $s^4-4p^2+4Lps+4=0$ is a quadratic in $p$, the quadratic formula yields
$$
p= \frac{Ls}{2} \pm \frac{1}{2}(s^2+2).
$$
If $p=s+\frac12(s^2+2)$, the polynomial in Equation (\ref{polywithL}) becomes $(t+1)^4$, so we always take
\begin{equation}\label{spequ}
p=s-\frac12(s^2+2).\end{equation}
We now have the polynomial
\begin{equation}\label{ourpoly}
F_s(t)= t^4 + (4s^3 - 4s^2 + 8s - 4)t^3 + (-6s^2 - 6)t^2 + 4t + 1,
\end{equation}
with the side condition that 
\begin{equation}\label{square}
3s^2-4s+4 \text{ is a square }
\end{equation}
(this is a rewriting of $s^2-4p$).
This condition immediately implies the following:
$$
s \text{ is even}, \quad s^2+2\equiv 2\pmod 4, \quad p\text{ is odd},
$$
where the last follows from (\ref{spequ}). Since $p$ is odd, we must have
$$
f \text{ and } g \text{ are odd}.
$$

The original parameters $f$ and $g$ are given by
\begin{equation}\label{fg}
f, g=\frac{s\pm \sqrt{3s^2-4s+4}}{2}.
\end{equation}
The choice of which is $f$ and which is $g$ does not have much significance,
but it affects the choice of generator of the Galois group in the following.

For future reference, we note the following consequence of $L= 2$:
\begin{lemma} \label{fglemma} If $L=2$, then
$$
\frac{s^2+2}{2} = \left(\frac{f+1}2\right)^2 +\left(\frac{g+1}2\right)^2, \qquad (f-g)^2=3s^2-4s+4.
$$
\end{lemma}
\begin{proof} The right side of the first equation is
$$
\frac{1}4\left(s^2-2p+2s+2\right)=\frac{1}4\left(2s^2+4\right)=\frac12\left(s^2+2\right),
$$
where the first equality uses Equation (\ref{spequ}).  The second equation follows from Equation (\ref{fg}). \end{proof}

Suppose now that  Equation (\ref{square}) holds. 
Then
$$
v^2=3s^2-4s+4,
$$
for some $v$,
which forces $v=2w$ and $s=2u$ for some $u, w$. This yields
$$
(3u-1)^2-3w^2=-2.
$$
The solutions are given by
\begin{equation}\label{Pell}
(3u-1)\pm w\sqrt{3}=(-1)^{n+1} (1+\sqrt{3})(2+\sqrt{3})^n,
\end{equation}
with $n\in \mathbb Z$ (the first choice of signs is arbitrary; the second sign is chosen 
in order to make the right side congruent to $-1$ mod $\sqrt{3}$).

This gives the following values of $s$:
$$
 4, \quad  -12, \quad  48, \quad  -176, \quad  660, \quad  -2460, \quad  9184, \quad  -34272.
$$
Every third value (4, $-176$, 9184, \dots) has $s^2+2$ divisible by $3^2$. The other values
listed yield squarefree values of $s^2+2$, although it is not known whether the sequence yields
infinitely many values of $s$ such that $s^2+2$ is squarefree. Questions of this type seem similar
to questions about squarefree Mersenne numbers, most of which are unsolved. For the first 100 values of $s$ (that is, for
$s$ arising from $1\le n\le 100$ in Equation (\ref{Pell})),
all values of $s^2+2$ or $(s^2+2)/9$ are squarefree.

The discriminant of the polynomial $F_s(t)$ in Equation (\ref{ourpoly}) is
\begin{equation}\label{disc}
256(3s^2-4s+4)^3(s^2+2)^3.
\end{equation}
Since $3s^2-4s+4$ is a square, the discriminant is a square times $s^2+2$. But $s^2+2$ is never a square,
so $k_s=\mathbb Q(\sqrt{s^2+2})\subseteq K_s$. Therefore, once we show that the Galois group is cyclic, we know
that $k_s$ is the unique quadratic subfield.

We first show that $F_s$ is irreducible, then identify the Galois action.

\begin{lemma} \label{roots} Let $|s|\ge 3$. The roots of $F_s(t)$ satisfy
\begin{gather*}
 r_1=-4s^3+4s^2-8s+4-(3/2)s^{-1}-(3/2)s^{-2}+ \theta_1 s^{-4}, \text{ with } 1\le \theta_1 \le 2,\\
r_2 = \frac{1+\sqrt{3}}{2} s^{-1} +\frac{3+\sqrt{3}}{6}s^{-2} -\frac{1}{3\sqrt{3}}s^{-3} + \theta_2 s^{-4}, \text{ with } -3/2 \le \theta_2\le -1/2,\\
r_3= (1/2)s^{-1} +(1/2)s^{-2} - \theta_3 s^{-4}, \text{ with } 0\le \theta_3\le 1,\\
r_4 = \frac{1-\sqrt{3}}{2} s^{-1} +\frac{3-\sqrt{3}}{6}s^{-2} +\frac{1}{3\sqrt{3}}s^{-3} + \theta_2 s^{-4}, \text{ with } -1/2 \le \theta_2\le 1/2.
\end{gather*}
\end{lemma}
\begin{proof}
Let $\overline{r}_1=-4s^3+4s^2-8s+4-(3/2)s^{-1}-(3/2)s^{-2}$. Substitute $\overline{r}_1+s^{-4}$  into $F_s(t)$.
The result is a degree 21 polynomial $P_1(s)$ divided by $s^{16}$. The real roots of $P_1(s)$ all have absolute value less than 2.1,
and $P_1(s)$ has a positive top coefficient. Since $P_1(s)$ is positive for large $s$ and does not change sign in the interval $(2.1, \infty)$,
we see that $P_1(s)>0$ for $s\ge 3$. Similarly, $P_1(s)<0$ when $s\le -3$. 

Now  substitute $\overline{r}+2s^{-4}$
into $F_s(t)$. The result is a degree 21 polynomial $Q_1(s)$ divided by $s^{16}$. 
The real roots of $Q_1(s)$ all have absolute value less than 1,
and $Q_1(s)$ has a negative top coefficient. 
It follows that $Q_1(s)<0$ for $s\ge 1$ and $Q_1(s)>0$ when $s\le -1$.

Fix $s$ with $s\ge 3$. Then $F_s(\overline{r}_1+s^{-4})>0>F_s(\overline{r}_1+2s^{-4})$.
Therefore, there is a zero $r_1$ of $F_s(t)$ that satisfies the stated conditions.
The case where $s\le -3$ is similar.

The proofs for $r_2, r_3, r_4$ are similar.\end{proof}

These expansions of $r_1$ and $r_3$ were found by letting $s=10^{100}$ and finding the roots of $F_s(t)$ numerically.
The coefficients of the expansions were then easy to deduce from the decimal expansions of the roots.
For $r_2$ and $r_4$, the expansions of $r_2+r_4$ and $(r_2-r_4)/\sqrt{3}$ had simple forms, and the above were obtained
from these. 

If we take $f=\left(s+\sqrt{3s^2-4s+4}\right)/2$ and $g=\left(s-\sqrt{3s^2-4s+4}\right)/2$, then the linear fractional transformation $M$ maps
$r_j$ to $r_{j+1}$. Given the approximation to $r_1$, we could obtain the other approximations 
from the action of $M$, but estimating the error terms would be harder.

\begin{lemma} Let $s\in \mathbb Z$. Then $F_s(t)$ is irreducible in $\mathbb Q[t]$.
\end{lemma}
\begin{proof} The only possible rational roots of $F_s$ are $\pm 1$. But $F_s(\pm 1)\ne 0$ when $s\in \mathbb Z$.
Therefore, $F_s$ does not have a linear factor, so, if it factors, it must have two quadratic factors,
and they must have integer coefficients. This means that $r_2+r_j\in \mathbb Z$ for some $j$.

Let $|s|\ge 10$, say. The cases with $|s|<10$ can be checked individually.
From Lemma \ref{roots}, we see that $r_2+ r_1$ is not an integer, so $r_2, r_1 $ cannot be the roots of a
quadratic factor. Also, $0<|r_2+r_3|<1$ and $0<|r_2+r_4|<1$, so these sums cannot be 
integers. Therefore, $F_s$ cannot factor into quadratic factors.    \end{proof}

\begin{lemma} Let $K_s$ be the splitting field of $F_s$ and assume $3s^2-4s+4$ is a square.
Then $\text{Gal}(K_s/\mathbb Q)$ is cyclic, and the linear fractional transformation $M$ gives the Galois action on
the roots of $F_s(t)$.\end{lemma}
\begin{proof} Since $3s^2-4s+4$ is a square, the parameters $f$ and $g$ exist (see Equation~(\ref{fg})).
Let $\theta$ be a root of $F_s$. Then $\theta$ satisfies Equation (\ref{norm1}). Let $\theta'=M\theta$, the result
of applying the linear fractional transformation $M$ to $\theta$. Since $M$ cyclically permutes
the factors in Equation (\ref{norm1}), we see that $\theta'$ also satisfies this equation, and therefore
$F_s(\theta')=0$. Therefore, $\mathbb Q(r_1)$ contains all the roots of $F_s$.
Since $F_s$ is irreducible, it follows that the Galois group
of $F_s$ is cyclic of order 4 and is generated by $M$.\end{proof}

\section{The discriminant}

\begin{proposition} Suppose that $s^2+2$ is squarefree and $3s^2-4s+4$ is a square. Then the discriminant of $K_s$ is $2^{8}(s^2+2)^3$.
\end{proposition}
\begin{proof}
The discriminant of the polynomial $F_s(t)$ is
$$
2^8 (s^2+2)^3 (3s^2-4s+4)^3 = 2^8(s^2+2)^3 (f-g)^6,
$$
where we have used Lemma \ref{fglemma} to obtain the second expression.
We need to show that the factor $(f-g)^6$ can be removed. Let $q$ be an odd prime dividing $f-g$.

Since $-4spL= s^4-4p^2+4$, we cannot have $f\equiv g\equiv 0\pmod q$. 

If $fg(f+g)\not\equiv 0\pmod q$ (equivalently, $f\not\equiv 0\pmod q$), then Equation (\ref{fgpoly})
becomes
$$
F_s(t) \equiv (t-1/f)^3 (t-f^3) \pmod q.
$$
Also,  Lemma \ref{fglemma} implies that 
\begin{equation}\label{cong}
0\equiv (f-g)^2 = 3s^2-4s+4=3(f+g)^2-4(f+g)+4\equiv 12f^2-8f+4\pmod q.
\end{equation}
We claim that the roots $1/f$ and $f^3$ of $F_s(t)\pmod q$ are distinct.
Suppose $1/f\equiv f^3\pmod q$. The resultant of $f^4-1$ and $12f^2-8f+4$ is  $2^{13}\cdot 3$,
so we must have $q=3$.

Equation (\ref{cong}) now tells us that $0\equiv 12f^2-8f+4\pmod 3$, so $f\equiv -1\pmod 3$. Since $f\equiv g$, we also have
$g\equiv -1\pmod 3$. Lemma \ref{fglemma} implies that $s^2+2\equiv 0\pmod 9$, contradicting the assumption that $s^2+3$ is squarefree.

Therefore, $1/f\not\equiv f^3\pmod q$.

Let $\mathfrak q$ be a prime of $K_s$ dividing $q$ and let $I$ be the inertia subgroup of $\text{Gal}(K_s/\mathbb Q)$ for $\mathfrak q$. 
If $q$ divides the discriminant of $K$, then $\sigma^2\in I$, where $\sigma$ 
generates $\text{Gal}(K_s/\mathbb Q)$.  Let $r_j\equiv f^3\pmod {\mathfrak q}$. Then the other three roots are congruent
to $1/f$ mod $\mathfrak q$. But $\sigma^2\in I$ means that $f^3\equiv r_j\equiv \sigma^2(r_j)\equiv 1/f \pmod {\mathfrak q}$, contradicting
the fact that $f\not\equiv 1/f^3\pmod q$.
Therefore, $q$ does not divide the discriminant of $K_s$, so $f-g$ contributes no odd prime factors to the discriminant of $K_s$.

We have proved that the discriminant $D$ of $K_s$ divides a power of 2 times $(s^2+2)^3$.

The subfield $k_s=\mathbb Q(\sqrt{s^2+2})\subset K_s$ has conductor $4(s^2+2)$, since $s^2+2$ is squarefree
and congruent to 2 mod 4. Let $\chi$ be a Dirichlet character of order 4 attached to $K_s$.
Then $\chi^2$ is the quadratic character attached to $k_s$. Since $\chi^2$ has conductor $4(s^2+2)$,
it follows that $\chi$ and $\chi^{-1}$ have conductor divisible by $4(s^2+2)$. The conductor-discriminant
formula implies that $D$ is divisible by $4^3(s^2+2)^3$. We have therefore proved that $D$ is a power of 2
times $((s^2+2)/2)^3$.

Since 2 ramifies in $k_s/\mathbb Q$ and $K_s/\mathbb Q$ is cyclic, 2 is totally ramified in $K_s/\mathbb Q$.
This means that $\chi$ is the product of a character of conductor 16 and a character of odd conductor. The
conductor-discriminant formula implies that $2^{11}$ is the exact power of 2 dividing $D$.  
Since $(s^2+2)^3$ contributes $2^3$, this completes the proof. \end{proof}

\noindent
{\bf Remark.} 
In order to determine the ramification and discriminant of $K_s$, we could consider the extension
$K_s(i)/\mathbb Q(i)$. 
Order the roots $r_1, r_2, r_3, r_4$, so that $Mr_j=r_{j+1}$. Let
$$
\rho=r_1+r_2i -r_3-r_4i.
$$
Computationally, it appears that
\begin{equation}\label{Kummer}
\rho^4 = -2^6i(f+gi)^8 \pi^3 \overline{\pi}, 
\end{equation}
where $\pi = (f+1)/2 - i(g+1)/2$. This would suffice to remove the factor $(3s^2-4s+4)^3$,
since $\sqrt[4]{\rho}$ generates the extension $K(i)/\mathbb Q(i)$, and since the odd parts of the discriminants
of $K_s(i)/\mathbb Q(i)$ and $K_s/\mathbb Q$ are equal. However, the verification of Equation (\ref{Kummer}) seems to be potentially quite
involved, which is why we had the incentive to find the above proof.

\section{Fundamental units}\label{fundamental}

The purpose of this section is to prove that $\pm 1$ and the roots of $F_s(t)$ generate the units of $K_s$.
Throughout this section, we assume that $3s^2-4s+4$ is a square.
 
\begin{lemma}\label{quadunit} Let $s\ne 0$ and suppose $s^2+2$ is squarefree. Then
$$\epsilon=-r_1r_3=
s^2+1+|s|\sqrt{s^2+2}$$ 
is the fundamental unit of the ring of integers of $\mathbb Q(\sqrt{s^2+2})$.\end{lemma}
\begin{proof} Since $s^2+2\not\equiv 1\pmod 4$, the fundamental unit is in $\mathbb Z[\sqrt{s^2+2}]$ (that is,
there is no 2 in the denominator). If $\epsilon_0=a+b\sqrt{s^2+2}$ is the fundamental unit, so $a, b>0$, then
$\epsilon_0^2>s^2+1+s\sqrt{s^2+2}=\epsilon$. Since $\epsilon$ is a power of $\epsilon_0$, we must have
$\epsilon=\epsilon_0$. 

From Lemma \ref{roots}, we see that
$ r_1r_3\approx -2s^2$. In particular, $1<-r_1r_3<\epsilon^2$, so $r_1r_3=-\epsilon$. \end{proof}

{\bf Note:} We did not need to know the ordering of the $r_j$ under the Galois group to obtain this result,
since $-r_1r_3$ is the only combination with the same approximate size as $\epsilon$. In fact, once we know this,
 if $\sigma$ is a generator of $\text{Gal}(K_s/\mathbb Q)$ then $\sigma^2$ maps $r_1$ to $r_3$,
and hence $\sigma$ or $\sigma^{-1}$
permutes the roots $r_1, r_2, r_3, r_4$
cyclically (that is, $r_j\mapsto r_{j+1}$). Of course, we know that $M=\sigma$ or $\sigma^{-1}$, depending
on the choice of signs in Equation (\ref{fg}).

\begin{proposition}\label{RvsD} Let $K$ be a totally real number field with $\text{Gal}(K/\mathbb Q)$ cyclic of order 4.
Let $R_K$ and $D_K$ be the regulator and discriminant of $K$. Let $k$ be the quadratic subfield
of $K$ and let $\epsilon$ and $d_k$ be the fundamental unit and discriminant of $k$.
Then
$$
\frac{1}{4}\log^2\left(\frac{D_K}{16 d_k^2}\right) \le \frac{R_K}{\log\epsilon}.
$$
If $d_k>150$, then
$$
\frac{1}{4}\log^2\left(\frac{D_K}{4.84 d_k^2}\right) < \frac{R_K}{\log\epsilon}.
$$
\end{proposition}
\begin{proof}
This result is implicit in \cite{BergeM}, \cite{Gras}, and \cite{WashQ}, for example. However, since it does not seem to
be explicit in the literature, we include the proof for the convenience of the reader.

The result is easily verified for $\mathbb Q(\zeta_{16})^+$, the maximal real subfield of the 16th cyclotomic field.
In all other cases, an odd prime ramifies in $K/\mathbb Q$, so we may
assume that $K/k$ ramifies at some prime of $k$ that divides an odd prime of $\mathbb Z$.

Let $E_K$ and $E_k$ be the units of $K$ and $k$, and let $\sigma$ generate $\text{Gal}(K/\mathbb Q)$. Then $E_K/E_k$ is a $\mathbb Z[i]$-module,
where $i$ acts as $\sigma$. We claim that this module is torsion-free. If $\alpha\in \mathbb Z[i]$ 
maps $u\in E_K$ into $E_k$, then so does $\text{Norm}(\alpha)$, so it suffices to show that
the $\mathbb Z$-torsion is trivial. 

If $u\in E_K$ and $u^n\in E_k$, then $(\sigma^2 u)^n = u^n$, so $\sigma^2 u = \pm u$,
since $k$ is real. Therefore, $\sigma^2(u^2)=u^2$, so $u^2\in E_k$. Therefore, $k(u)/k$ ramifies
at most at the primes above 2, hence is trivial since $K/k$ is assumed to ramify at an odd prime.
Therefore, $u\in E_k$, so $E_K/E_k$ is torsion-free.

Since $E_K/E_k$ has $\mathbb Z$-rank 2, it has $\mathbb Z[i]$-rank 1, so there is a unit $\eta$
that generates it as a $\mathbb Z[i]$-module. Let $\eta^{1+\sigma^2}=\delta\in E_k$, and let
$\eta'=\sigma(\eta)$. Then $\{\pm 1, \epsilon, \eta, \eta'\}$ generates $E_K$ as a $\mathbb Z$-module.

A calculation shows that the regulator of $K$ is
$$
R_K=2 \log(\epsilon)\left( (\log |\eta| -\frac12\log |\delta|)^2 + (\log |\eta'| +\frac12\log |\delta|)^2\right).
$$
The different of $K/k$ divides $\eta-\sigma^2(\eta)=\eta-\delta/\eta$, so the discriminant of $K/k$
divides 
$$\text{Norm}_{K/k}(\eta-\delta/\eta)=-(\eta-\delta/\eta)^2.$$
Since $D_K=d_k^2 \text{ Norm}_{k/\mathbb Q}(D_{K/k})$, we find that
$$
D_K/d_k^2 \text{ divides } (\eta-\delta/\eta)^2(\eta'\pm1/(\delta\eta'))^2 \le \left(x+1/x\right)^2\left(y+1/y\right)^2,
$$
where $x=\text{Max}(|\eta|/|\delta|^{1/2}, \, |\delta|^{1/2}/|\eta|)$ and
$y=\text{Max}(|\eta'||\delta|^{1/2}, \, 1/|\delta|^{1/2}|\eta'|)$.

The conductor-discriminant formula implies that $d_k^3\mid D_K$. If $1\le x, y\le \sqrt{10}$, then
$$
d_k\le (x+1/x)^2(y+1/y)^2 <150.
$$
Therefore, if $d_k>150$ then at least one of $x, y$ is larger than $\sqrt{10}$. If
$x>\sqrt{10}$, then $x+1/x< 1.1x$. If $1\le x\le \sqrt{10}$, then $x+1/x\le 2x$. Therefore,
$$
\left(x+1/x\right)\left(y+1/y\right) < 2.2xy,
$$
so
\begin{align*}
\log (D_K/d_k^2) &\le \log \left(x+1/x\right)^2\left(y+1/y\right)^2 < 2\left(\log 2.2 +\log x  + \log y\right)\\
&\le 2\left(\log 2.2 +\sqrt{2} \left(\log^2 x + \log^2 y\right)^{1/2}\right) \text{ (Cauchy-Schwarz)}\\
&= 2\left(\log 2.2 +(R_K/\log \epsilon)^{1/2}\right).
\end{align*}
This yields the last statement of the proposition. Note that by increasing the lower bound for $d_k$, we could
replace $4.84$ by any number larger than $4$.

If we do not require $d_k>150$, then a slightly simpler argument works (we thank St\'ephane Louboutin for pointing this out):
We have $x, y\ge 1$, so $x+1/x\le x+1\le 2x$ and similarly for $y$. Therefore,
$$
\left(x+1/x\right)\left(y+1/y\right) \le  4xy.
$$
The above argument yields 
$$
\frac{1}{4}\log^2\left(\frac{D_K}{16 d_k^2}\right) \le \frac{R_K}{\log\epsilon}.
$$
This inequality suffices for our purposes.
\end{proof}

\begin{lemma}\label{indexnot234} Assume $s^2+2$ is squarefree. 
Let $E$ be the units of $K_s$ and let $U$ be the subgroup generated by $\pm 1$ and
$r_1, r_2, r_3, r_4$. Then $[E: U]\ne 2, 3, 4, 6, 7, 8$.
\end{lemma}
\begin{proof}
Let $E_k$ be the units of $\mathbb Q(\sqrt{s^2+2})$. By Lemma \ref{quadunit}, $E_k\subset U$.
Let $\overline{E} = E/E_k$ and $\overline{U}=U/E_k$. Then $[\overline{E} : \overline{U}]=[E : U]$. 
As in the proof of Proposition \ref{RvsD}, $\overline{E}\simeq \mathbb Z[i]$.  The subgroup
$\overline{U}$ is an ideal of $\mathbb Z[i]$ under this isomorphism.

If $[E : U]=2$, then there exists 
$\eta$ that generates $\overline{E}$ as a $\mathbb Z[i]$-module and such that $(1+i)\eta\equiv r_1$ mod $E_k$.
This means that there exists $\delta\in E_k$ such that $\eta^{1+\sigma}=r_1\delta$. Since $\sigma^2$ fixes $\delta$,
we have 
$$
\pm 1=\eta^{1+\sigma+\sigma^2+\sigma^3}=r_1r_3\delta^2=-\epsilon\delta^2.
$$
But $\epsilon$ is the fundamental unit of $\mathbb Q(\sqrt{s^2+2})$, so this is impossible.

Now suppose $[E : U]=4$. The only ideal of index 4 in $\mathbb Z[i]$ is $(2)$, so there exists $\eta\in E$ such that 
$\eta^2=r_1\delta$, with $\delta\in E_{k_s}$. Taking norms to $k_s$ yields $N(\eta)^2=-\epsilon\delta^2$. This
is impossible because $\epsilon$ is the fundamental unit of $k_s$.

If $[E : U]=8$, then there exists $\eta$ such that $\eta^{2(1+\sigma)}=r_1\delta$, with $\delta\in E_{k_s}$. Taking norms from $K_s$ to $k_s$,
we find 
$$
1=\left(\eta^{1+\sigma+\sigma^2+\sigma^3}\right)^2=r_1^{1+\sigma^2}\delta^{1+\sigma^2}=-\epsilon \delta^2,
$$
which is impossible.

Finally, every ideal in $\mathbb Z[i]$ has index that is a norm from $\mathbb Z[i]$ to $\mathbb Z$. Since 3, 6, 7 are
not norms, $[E : U]\ne 3, 6, 7$.
 \end{proof}

We can now show that $\pm 1$ and $r_1, r_2, r_3, r_4$ generate the units of $K_s$.
We assume that $|s| \ge 10^5$. The cases where $|s|<10^5$ can be checked individually.

Let $R$ be the regulator of $K_s$ and let $R'$ be the regulator for the group $U$ in Lemma \ref{indexnot234}. 
Then $R'/R = [E: U]$.
We have 
$$
R'=\pm\det \begin{pmatrix} \log |r_1| & \log |r_2| & \log |r_3| \\ \log |r_2| & \log |r_3| & \log |r_4| \\ \log |r_3| & \log |r_4| & \log |r_1| 
\end{pmatrix}.
$$
This equals (see \cite[Lemma 5.26(c)]{WashCycl})
\begin{gather*}
\frac14 \left(\log |r_1| +i \log |r_2| - \log |r_3| - i\log |r_4|\right)\left(\log |r_1| - \log |r_2| + \log |r_3| - \log |r_4|\right)\\
\times
\left(\log |r_1| -i \log |r_2| - \log |r_3| + i\log |r_4|\right)\\
= \frac14 \left(\log^2|r_1/r_3| + \log^2 |r_2/r_4|\right) \left(2\log \epsilon\right).
\end{gather*}
Lemma \ref{roots} implies that 
\begin{equation}\label{9ineq}
R'/\log \epsilon \le \frac12 \left(\log^2(9s^4) + \log^2(4)\right) < 9\log^2 |s|
\end{equation}
when $|s| \ge 10^5$. Also,
$$
\log^2\left(\frac{D_{K_s}}{16 d_{k_s}^2}\right) = \log^2(s^2+2) \ge 4\log^2 |s|
$$
when $|s|\ge 1$. 
Proposition \ref{RvsD} implies that
$$
\log^2 |s| \le \frac{R'/[E:U]}{\log \epsilon} < \frac{9}{[E:U]}\log^2 |s|.$$
Therefore, 
$$
[E:U] < 9.
$$
By Lemma \ref{indexnot234}, $[E:U]\ne 2, 3, 4, 6, 7, 8$, so $[E: U]=1$ or $5$. 
This completes the proof of Theorem \ref{main}.

When $s\equiv 4, 5\pmod 9$, we have $s^2+2\equiv 0\pmod 9$, so $s^2+2$ is not squarefree. However,
if $(s^2+2)/9$ is squarefree, then the effect on $D_{K_s}$ is small enough that the
above proof shows that we still obtain either the full group of units or a subgroup of index 5, except when $s=4$ (where the index $[E_{K_s} : U] = 40$).
Computational evidence suggests that the index 5 case does not occur.

Since the units of $K_s$ are fairly small, we expect the class numbers to be large. Here is a table of the class groups for the first few values of $s$
(for the class groups, $a\times b$ means $\mathbb Z/a\mathbb Z \times \mathbb Z/b\mathbb Z$).
These calculations, and others in this paper, were  done using GP/PARI \cite{pari}.

\begin{table}[h]
\begin{tabular}{c|c}
$s$ & class group of $K_s$ \\
\hline
4 & 1\\
$-12$ & 4\\
48 & $4\times 4\times 4$\\
$-176$ & $60\times 5$ \\
660 & $260\times 20\times 5$ \\
$-2460$ & $81120\times 4\times 2\times 2$ \\
9184 & $115500\times 28$\\
$-34272$ &  $25104840\times 30\times 3$\\
$127908$ & $924437696\times 4\times 4\times 4$\\
$-477356$ &   $1332657200\times 20\times 2\times 2$\\
$1781520$ &  $28009347406480\times 2$\\
$-6648720$ &  $25020857770200\times 20\times 4\times 2$\\
$24813364$ &  $3937737813077376\times 4\times 2$\\
$-92604732$ & $21266991873333180\times 20\times 4\times 4$\\
$345605568$ & $4788485135078294496\times 12\times 6$
\end{tabular}
\end{table}
For all of the examples in the above table, $-1$ and the roots of the polynomial $F_s(t)$ generate the full group of units
for the ring of algebraic integers of the corresponding field, except for $s=4$, where they generate a subgroup of index 40.
In this case, $s^2+2=18$ is not squarefree, and the unit $17+4\sqrt{18}=(s^2+1) + s\sqrt{s^2+2}$ is the fourth power 
of the fundamental unit $1+\sqrt{2}$ of the quadratic subfield.

The growth of the class number can be made explicit. Since $K_s/k_s$ is ramified at 2, the class number of $k_s$ divides the class number
of $K_s$. Since the class number of $\mathbb Q(\sqrt{s^2+2})$ goes to $\infty$ as $s\to \infty$ (this can be made explicit with at most one 
possible exception;
see \cite{MollWill}), the class number of $K_s$ also goes to $\infty$. However, the potential effect of Siegel zeroes for quadratic fields
can be overcome since we have a quartic field, so we obtain a result with no exceptions.

\begin{proposition}\label{classbd} Assume $s^2+2$ is squarefree and $3s^2-4s+4$ is a square. Let $h$ be the class number of $K_s$ and let
$h_2$ be the class number of $k_s$. If $|s|\ge 10^5$, then
$$
\frac{h}{h_2} \ge  \frac{1}{450} \;\frac{s^2+2}{\log^2(|s|) \log^2(8(s^2+2)/\pi)}.
$$
\end{proposition}
\begin{proof} Let $\chi$ be a quartic Dirichlet character associated with $K_s$. Then
$$
\frac{8hR}{\sqrt{D_{K_s}}} = L(1, \chi) L(1,\chi^2) L(1, \chi^3).
$$
We have
$$
L(1, \chi^2)=\frac{2h_2\log \epsilon}{\sqrt{d_{k_s}}}.
$$
Louboutin \cite{Lou} has shown that for a non-quadratic primitive Dirichlet character $\chi$ of conductor $c\ge 90000$,
$$
| L(1,\chi)|\ge \frac{1}{10\log(c/\pi)}.
$$
Therefore,
$$
h\ge \frac{1}{400}\left(\frac{D_{K_s}}{d_{k_s}}\right)^{1/2}\;\frac{\log \epsilon}{R} \frac{1}{\log^2(c/\pi)},
$$
where $c=8(s^2+2)$. By Equation (\ref{9ineq}), $R/\log \epsilon\le R'/\log\epsilon \le 9\log^2|s|$.
Putting these together, we obtain the result. \end{proof}
It follows that $h>1$ when $|s|\ge 10^5$. Since we have computed the class number of $K_s$ for $|s|<10^5$, 
we obtain as a corollary that $K_s$ has class number 1 only for $s=4$. 

St\'ephane Louboutin pointed out to us the following:  Let $\chi$ be an even
Dirichlet character of conductor $f>1$. Then there is the upper bound (see \cite{Lou1})
$$
| L(1,\chi)| \le \frac12\left(c+ \log f\right),
$$
where $c= 2+\gamma - \log(4\pi)\approx 0.04619\cdots$ (and where $\gamma = .577\cdots$).
Let $K$ be a totally real quartic field of conductor $f_K$ corresponding to the quartic character $\chi_K$,
and let $k$ be its quadratic subfield. Then
$$
h_K/h_k = \frac{f_K}{4R/\log\epsilon} | L(1, \chi_K)|^2 \le \left(\frac{c+\log(f_K)}{2\log(f_K/16)}\right)^2 f_K,
$$
where we have used Proposition \ref{RvsD} to bound $R/\log\epsilon$.
Therefore, $h_K/h_k <f_K$ when $f_K>256 e^c \approx 268.01$, and a quick search using GP/PARI \cite{pari}
shows that $h_K<f_K$ also when $f_K\le 268$. Since $f_K=8(s^2+2)$ in the situation of Proposition \ref{classbd},
we see that the lower bound estimate given there is, up to log factors, the correct order of magnitude.

\section{Back to cubics}

The construction of $F_s(t)$ was inspired by the cubic case. See \cite{Balady}, which starts with the element
$$
\begin{pmatrix} f & -h\\ (f^2+g^2-fg)/h & -g\end{pmatrix} \in \text{PGL}_2(\mathbb Q),
$$
where $f, g, h$ can be taken to be distinct integers with $h\ne 0$. 
Assume that the associated linear transformation gives the action of a Galois group on
a number $\theta$, and assume that $\theta$ has norm 1 to $\mathbb Q$. Then we obtain
the equation
$$
\theta \cdot \frac{f\theta -h}{\theta (f^2+g^2-fg)/h - g} \cdot \frac{g\theta -h}{\theta (f^2+g^2-fg)/h - f} = 1,
$$
which can be rearranged to say that $\theta$ is a root of
$$
t^3 + \frac{3(f^2+g^2-fg)-\lambda h(f+g)}{h^2} t^2 + \lambda t-1,
$$
where $\lambda = (f^3+g^3+h^3)/(fgh)$.
If $h=1$ and $\lambda\in \mathbb Z$, we have a polynomial with integral coefficients. Therefore, we want
integral points $(x, y, 1)$ on the elliptic surface $x^3+y^3+1=\lambda xy$, which is the elliptic modular surface $X(3)$.

Following the lead of the quartic case, we look at the  singular fiber $\lambda = 3$. If $f$ is an integer,
we need $g$ to be a root of
$$
X^3-3fX+f^3+1=(X+f+1)(X^2-(f+1)X+f^2-f+1).
$$
When $f\ne 1$, the second factor is irreducible over $\mathbb Q$, so we must have $g=-f-1$. 
This yields the family of polynomials
$$
G_f(t)=t^3+(9f^2+9f+6)t^2+3t-1,
$$
which could be regarded as the analogue of our family of quartic polynomials. This cubic family appears in \cite{Kishi}.

 Let $r$ be a root of the polynomial
$x^3+(3f+3)x^2+3fx-1$. As pointed out in \cite{Kishi}, $-r^2-r$ is a root of $G_f(t)$. This can easily be verified by
representing $r$ by the matrix $A=\begin{pmatrix} 0 & 0 & 1\\ 1 & 0 & -3f-3\\ 0 & 1 & -3f\end{pmatrix}$,
then computing the characteristic polynomial of $-A^2-A$.
The fields obtained from the polynomials $G_f(t)$ are some of Shanks's ``simplest cubic fields.''
However, $-1$ and the roots of $G_f(t)$ generate a subgroup of index 3 in the group of units generated by the roots
of the polynomial $x^3+(3f+3)x^2+3fx-1$. 

In the cubic case, there are also many families of polynomials corresponding to curves on the surface $X(3)$ that do not
lie in the singular fiber $\lambda=3$ (see \cite{Balady}). 
It would be interesting to find similar families in the quartic case.

\subsection*{Acknowledgment.} The authors thank St\'ephane Louboutin for several helpful suggestions.

\end{document}